\documentclass{IEEEtran}
\usepackage{cite}
\usepackage{amsmath,amssymb,amsfonts}
\usepackage{algorithm}
\usepackage{algpseudocode}
\usepackage{graphicx}
\usepackage{textcomp}
\def\BibTeX{{\rm B\kern-.05em{\sc i\kern-.025em b}\kern-.08em
    T\kern-.1667em\lower.7ex\hbox{E}\kern-.125emX}}
\usepackage{mathrsfs}
\usepackage{dsfont}
\newtheorem{thm}{\bf Theorem}[section]
\newtheorem{cor}[thm]{\bf Corollary}
\newtheorem{lem}[thm]{\bf Lemma}
\newtheorem{rem}[thm]{\bf Remark}
\newtheorem{ass}[thm]{\bf Assumption}
\newtheorem{prop}[thm]{\bf Proposition}
\newtheorem{prob}[thm]{\bf Problem}
\newtheorem{deff}[thm]{\bf Definition}
\newtheorem{eg}[thm]{\bf Example}
\newtheorem{claim}[thm]{\bf Claim}

\newcommand{\eps}{\varepsilon}
\newcommand{\dt}{\delta t}
\newcommand{\ups}{\upsilon}
\newcommand{\hh}{\mathcal{H}}
\newcommand{\as}{\mathcal{A}}
\newcommand{\ls}{\mathcal{L}}
\newcommand{\vs}{\mathcal{V}}
\newcommand{\cs}{\mathcal{C}}
\newcommand{\os}{\mathcal{O}}
\newcommand{\hb}{\mathscr{H}}
\newcommand{\hs}{\mathcal{H}}
\newcommand{\R}{\mathbb{R}}
\newcommand{\B}{\mathbb{B}}
\newcommand{\N}{\mathbb{N}}
\newcommand{\ob}{\mathbf{\Omega}}
\newcommand{\qf}{\mathfrak{q}}
\newcommand{\ts}{\mathcal{T}}
\newcommand{\I}{\mathcal{I}}
\newcommand{\cf}{\mathfrak{c}}
\newcommand{\ef}{\mathfrak{er}}
\newcommand{\pf}{\mathfrak{p}}
\newcommand{\xb}{\mathbf{x}}
\newcommand{\yb}{\mathbf{y}}
\newcommand{\zb}{\mathbf{z}}
\newcommand{\tx}{\tilde{x}}
\newcommand{\ty}{\tilde{y}}
\newcommand{\tz}{\tilde{z}}
\newcommand{\tu}{\tilde{u}}
\newcommand{\tf}{\tilde{f}}
\newcommand{\tg}{\tilde{g}}
\newcommand{\tv}{\tilde{v}}
\newcommand{\tb}{\mathbb{T}}
\newcommand{\hx}{\hat{x}}
\newcommand{\om}{\Omega}
\newcommand{\w}{\varpi}
\newcommand{\dom}{\operatorname{dom}}
\newcommand{\lmax}{L_{0}}
\newcommand{\zbar}{\bar{z}}
\newcommand{\lya}{\lambda^\gamma}
\newcommand{\xx}{\mathcal{X}}
\newcommand{\uu}{\mathcal{U}}
\newcommand{\huu}{\hat{\mathcal{U}}}
\newcommand{\hrr}{\hat{\mathcal{R}}}
\newcommand{\hphi}{\hat{\phi}}
\newcommand{\hPhi}{\hat{\Phi}}
\newcommand{\hd}{\hat{d}}
\newcommand{\rr}{\mathcal{R}}
\newcommand{\uf}{\mathfrak{u}}
\newcommand{\vf}{\mathfrak{v}}
\newcommand{\psp}{\mathfrak{P}}
\newcommand{\ck}{\mathfrak{C}}
\newcommand{\dk}{\mathfrak{d}}
\newcommand{\hdk}{\hat{\mathfrak{d}}}
\newcommand{\pk}{\mathfrak{p}}
\newcommand{\tta}{\mathcal{P}}
\newcommand{\kk}{\mathfrak{K}}
\newcommand{\xk}{\mathfrak{x}}
\newcommand{\yk}{\mathfrak{y}}

\newcommand{\rbar}{\bar{r}}
\newcommand{\rup}{\check{r}}
\newcommand{\rhat}{\hat{r}}

\newcommand{\ub}{{\mathbf{u}}}
\newcommand{\eb}{{\mathbf{e}}}
\newcommand{\hxb}{\hat{\mathbf{x}}}
\newcommand{\hub}{\hat{\mathbf{u}}}
\newcommand{\ee}{\mathcal{E}}
\newcommand{\ex}{\tau_{\operatorname{ex}}}

\newcommand{\bp}{\bar{p}}
\newcommand{\bq}{\bar{q}}

\newcommand{\K}{\operatorname{K}}
\newcommand{\trans}{\mathcal{T}}
\newcommand{\transt}{\widetilde{\mathcal{T}}}
\newcommand{\rrs}{\mathscr{R}}
\newcommand{\rrk}{\mathfrak{R}}
\newcommand{\W}{\widetilde{W}}
\newcommand{\dd}{\mathcal{D}}
\newcommand{\ddc}{\mathcal{D}_{\operatorname{con}}}
\newcommand{\init}{\mathcal{X}_0}
\newcommand{\inte}{\operatorname{Int}}
\newcommand{\st}{{\tau_\dd}}
\newcommand{\aba}{\oball_R(A)}
\newcommand{\supx}{\sup\limits_{X\in\soln(x,W)}}
\newcommand{\infx}{\inf\limits_{X\in\soln(x,W)}}
\newcommand{\cb}{C_{\operatorname{BL}}}
\newcommand{\ra}{\rightarrow}
\newcommand{\set}[1]{\left\{#1\right\}}
\newcommand{\norm}[1]{\left| #1 \right|}
\newcommand{\abs}[1]{\left\vert #1 \right\vert}
\newcommand{\bl}[1]{\left\|#1\right\|_{\operatorname{BL}}}
\newcommand{\qv}[1]{\left\langle\langle#1\right\rangle\rangle}
\newcommand{\lip}[1]{\left\|#1\right\|_{\operatorname{Lip}}}
\newcommand{\lips}[1]{\left\|#1\right\|_{\operatorname{BL,\;s}}}
\newcommand{\Qs}{\mathcal{Q}}
\newcommand{\ball}{\overline{\mathcal{B}}}
\newcommand{\oball}{\mathcal{B}}
\newcommand{\sig}{\varsigma}
\newcommand{\tr}{\operatorname{tr}}
\newcommand{\Ss}{\mathcal{S}}
\newcommand{\LLk}{\mathfrak{L}}
\newcommand{\normr}[1]{\left\|#1\right\|_{\operatorname{R}}}

\newcommand{\Qed}{\hfill$\diamond$} 
\newcommand{\Qede}{\hfill\diamond} 
\usepackage{xcolor}
\newcommand{\ymmark}[1]{{\color{black} #1}}
\newcommand{\ymmarkr}[1]{{\color{black} #1}}

\newcommand{\BlackBox}{\rule{1.5ex}{1.5ex}}    
\newcommand{\pfbox}{\hfill\BlackBox\\[1mm]}  
\newenvironment{proof}{\par\noindent{\bf Proof:\ }}{\hfill\BlackBox\\[1mm]}

\usepackage{hyperref}
 \hypersetup{
    colorlinks=true,
    linkcolor=blue,
    filecolor=magenta,      
    urlcolor=cyan,
    pdftitle={Overleaf Example},
    pdfpagemode=FullScreen,
    }
    \usepackage[misc]{ifsym}

\usepackage{dsfont}
\IEEEoverridecommandlockouts  

\begin{document}
\title{Online Learning and Control Synthesis 
for 
Reachable Paths of Unknown Nonlinear Systems} 
\author{Yiming Meng, Taha Shafa, Jesse Wei, Melkior Ornik \IEEEmembership{Senior Member, IEEE}
\thanks{This research was supported by NASA under grant numbers 80NSSC21K1030 and 80NSSC22M0070, as well as by the Air Force Office of Scientific Research under grant number FA9550-23-1-0131.}

\thanks{
Yiming Meng is with  the
Coordinated Science Laboratory, University of Illinois Urbana-Champaign,
Urbana, IL 61801, USA.
        {\tt\small ymmeng@illinois.edu}.
}

\thanks{Taha Shafa and Jesse Wei are with the Department of Aerospace Engineering , University of Illinois Urbana-Champaign,
Urbana, IL 61801, USA.
        {\tt\small tahaas2, jwei28@illinois.edu}.}

\thanks{
Melkior Ornik is with the Department of Aerospace Engineering and the
Coordinated Science Laboratory,
University of Illinois Urbana-Champaign,
Urbana, IL 61801, USA. {\tt\small  mornik@illinois.edu}.}
}

\maketitle

\begin{abstract}
In this paper, we present a novel method to drive a nonlinear system to a desired state, with limited a priori knowledge of its dynamic model: local dynamics at a single point and the bounds on the rate of change of these dynamics. 
This method synthesizes control actions by utilizing locally learned dynamics along a trajectory, based on data available up to that moment, and known proxy dynamics, which can generate an underapproximation of the unknown system's true reachable set. An important benefit to the contributions of this paper is the lack of knowledge needed to execute the presented control method. We establish sufficient conditions  to ensure that a controlled trajectory reaches a small neighborhood of any provably reachable state within a short time horizon, with precision dependent on the tunable parameters of these conditions. 
\end{abstract}


\section{Introduction}\label{sec:introduction}
Systems across domains operate with limited information, such as uncertainties arising from an insufficient understanding of system transitions and external forces. 

In this paper, we focus on the situation where the nonlinear system is partially unknown, with our knowledge limited to its local dynamics at a single point and the bounds on the rate of change of these dynamics. Based on this restricted information, we aim to implement the following pipeline for the system. 
 First, we identify a set of states, known as the Guaranteed Reachable Set (GRS), that the unknown system can provably reach within a given timeframe from the point of known information, using underapproximation proxy dynamics \cite{shafa2022maximal, Shafa23Reachability}. Then, we specify a  state on the boundary of the GRS and synthesize a controller, enabling the partially unknown  system to approach the vicinity of this state. Although the works in \cite{shafa2022maximal, Shafa23Reachability} provide a systematic method for estimating the GRS based on set-valued mapping analysis, they lack the capability to conversely identify a control signal that maneuvers the true system to reach a specified region within the GRS. Therefore, we concentrate on addressing the question of how to design control for the reachability task mentioned above.



Motivated by scenarios where an adverse event can cause significant changes to the system dynamics, there exists a substantial body of work in the realm of control for systems with  uncertainties in the dynamic model. 
Data-driven and neural network control methods \cite{cheng2021adaptive, bianchin2021data,taylor2021towards,truong2023novel} can be utilized for applications that include computing steady-state transfer functions in real-time 
or synthesizing controllers for nonlinear systems with dynamic modeling uncertainties. 
Other methods involve learning a Gaussian process model \cite{awan2023formal, kocijan2004gaussian} for black-box identification of nonlinear systems. Methods outside of identification \cite{chen2022intelligent,jin2023prediction,mauroy2019koopman, zeng2022sampling,meng2023learning, meng2024resolvent, zeng2024data} and machine learning methods \cite{shmalko2021control,meng2024physics} utilize control tools like control barrier functions \cite{ames2019control} and more classical adaptive control methods \cite{seto1994adaptive,jiang1998design,astolfi2008nonlinear},  
but are limited to handling less significant uncertainties.  While all of these methods have their merits, they all require the use of significantly more information, such as an existing model with uncertainties or system data. 

In contrast, the novel control method we present for online controller synthesis relies solely on discrete-time, up-to-date historical data from a single trajectory run and a derived underapproximation proxy control system based on presumed model information \cite{shafa2022maximal, Shafa23Reachability}.   It is worth noting that the work \cite{ornik2019control} is capable of learning dynamics and potentially achieving the specified control tasks, such as reachability and safety, using the same data. However, a `goodness function'—encoding side information about the system, such as physical laws—must be carefully designed to ensure that the controlled trajectory moves in a nearly optimal direction. 
The key difference highlighted in this paper is that we can utilize the knowledge of the GRS and the proxy system that generates it for control design. This framework relaxes the requirement for potential side information from the unknown system and provides a reachability guarantee.



The outline of this paper is as follows. In Section \ref{sec: pre}, we discuss the preliminary knowledge we assume, which is  in line with assumptions in \cite{shafa2022maximal,Shafa23Reachability}. In Section \ref{sec: proxy_preliminaries}, we show the necessary properties of the underapproximated proxy control system, \ymmarkr{particularly in identifying a path to a provably reachable desired state, which serves as a reference trajectory for the original system to follow. In Section \ref{sec: control}, we prove that    the proposed framework enables the true system's state to track the reference path and converge to the desired state.} 
We then present an algorithm for applying this control method.  Lastly, we conduct a case study to investigate the effectiveness of the proposed algorithm, as detailed in Section \ref{sec: case}.


\textbf{Notation:} 
We denote the Euclidean space by $\R^d$ for $d>1$.  We denote $\R$ the set of real numbers, and $\R_{\geq 0}$ the set of nonnegative real numbers.  
The closed ball of radius $r$ centered at $x\in\R^d$ is denoted by $\oball^d(x;r):=\{y\in\R^d: |y-x|\leq  r\}$, where $|\cdot|$ is the Euclidean norm. 
For a given set $A\subseteq\R^d$, 
$\inte(A)$ denotes its interior, and
$\partial A$ denotes its boundary.  For a closed set $A\subseteq \R^d$ and $x\in\R^d$, we denote the distance from $x$ to $A$ by $|x|_A=\inf_{y\in A}|x-y|$. 
For a matrix $M$, $\|M\|$ denotes its Euclidean norm: $\|M\|=\sup_{|v|=1}|Mv|$, $M^\dagger$ denotes the Moore-Penrose
pseudoinverse, $\operatorname{Im}(M)$ denotes the image. For matrices $M$ and $N$, \ymmarkr{we say  $M\in\operatorname{Imm}(N)$ if all columns of $M$ lie in $\operatorname{Im}(N)$}. 



\section{Preliminaries}\label{sec: pre}

Let  the admissible set of inputs be $\uu=\oball^m(0;1)$, 
which is a common setting in reachability analysis \cite{margaliot2007reachable,vinter1980characterization}. 
We consider the following unknown 
dynamical system defined by
\begin{small}
    \begin{equation}\label{E: sys}
    \dot{\xb}(t) = f(\xb(t)) + G(\xb(t))\ub(t), \;\;\xb(0) = x_0, 
\end{equation}
\end{small}

\noindent where for all $t\geq 0$, $\xb(t)\in\R^d$; $\ub:\R_{\geq 0}\ra \uu$.  The mappings $f:\R^d\ra\R^d$, and $G:\R^d\ra\R^{d\times m}$ are 
Lipschitz continuous with Lipschitz constants $L_f, L_G\geq 0$. We introduce $\lmax:=\max\set{L_f, L_G}$ for future references. 

For any initial condition $x_0\in\R^d$, we denote by $\phi_{\ub}(\cdot, x_0):[0, \infty)\ra\R^d$ the controlled flow map (solution) under $\ub$, and by $\Phi(x_0; \uu)$ the set of solutions. We simply use  $\phi_{\ub}(\cdot)$ if we do not emphasize the initial
condition. For any $T\geq 0$, \ymmarkr{the set of states reached by solutions starting from $x_0$ at time $T$ is defined as $\rr^T(x_0):= \set{\phi_\ub(T):  \phi_\ub(\cdot)\in\Phi(x_0;\uu)}$, and we also define the reachable set as $\rr^{\leq T}(x_0):=\cup_{t\in[0, T]}\rr^t(x_0)$.} 

For the reachability analysis and control synthesis, we propose the following hypotheses that are consistent with those presented in \cite[Section II.A]{Shafa23Reachability}:   
(H1) $f$ and $G$ are of the form $f(x) = Bh(x)$
and $G(x) = BH(x)$ where $B\in\R^{d\times m}$ is a constant matrix, 
$h: \R^d\ra\R^m$, and $H:\R^d\ra\R^{m\times m}$ such that $H(x_0)$  is
invertible; (H2)   $L_f$ and $L_G$ are known, as well as values
$f(x_0)$ and $G(x_0)$ such that $G(x_0)\neq 0$. We also assume knowledge of
$\operatorname{Im}(B)$ and $\operatorname{Imm}(B)$. However, we do not need to know the value of
$B$ exactly.

For convenience, we denote by $\ddc$  the set of all pairs $f$ and $G$ consistent with H1 and H2  for the same $L_f$ and $L_G$. \ymmarkr{We also denote $C:=\|G(x_0)\|\|G(x_0)^\dagger\|$}. 

\begin{rem}
    The settings are motivated by situations where damage to a  system can  alter its dynamics, with these changes captured by system \eqref{E: sys}. In practice, an online approximation of  $f(x_0)$ and $G(x_0)$ can be achieved with sufficient accuracy 
\cite{el2023online,ornik2019control}. However, for clearer illustration, we assume that the information at $x_0$ is perfectly known, allowing us to further explore the system’s remaining capabilities.  

Additionally,  the recent work in \cite{shafa2024guaranteed} extends \cite{shafa2022maximal, Shafa23Reachability} to Riemann manifolds, with applications for underactuated systems. This paper presents a control synthesis technique based on (H1) and (H2), with plans to generalize it based on more relaxed conditions for $f$ and $G$ given insights from \cite{shafa2024guaranteed}. 
 \Qed 
\end{rem}

\section{\bf Underapproximated Proxy Control System and Its Reachable Paths}\label{sec: proxy_preliminaries}

With limited knowledge about the system dynamics, we attempt to construct a control signal to reach a small neighborhood of a specified point  on the boundary of 
  the GRS, which is contained within the true reachable set. In this section, 
we revisit the underapproximated proxy control system for GRS evaluation \cite{Shafa23Reachability}, which can be viewed as supplementary information for the above reachability control problem. Particularly, we highlight the connections to the actual system \eqref{E: sys}. Then, we investigate properties related to the reachable paths and their corresponding control signals for the proxy system. 

\subsection{The Proxy System}\label{sec: proxy}
The underapproximated proxy control system is obtained through a connection to \eqref{E: sys} as shown in \cite[Theorem 1]{Shafa23Reachability}. We rephrase the statement as follows. 
 \begin{thm}\label{thm: track}
     Let $\uu$, $L_f$, and $L_G$ be given. Then,   there exists a $u\in\uu$ satisfying     \begin{small}
               \begin{equation}
         k\hat{u}=f(x) -f(x_0) + G(x)u 
     \end{equation} 
      \end{small} 
      
      \noindent for all  $x\in\B:=\{x: |x|\leq \|G(x_0)^\dagger\|^{-1}/(L_f+L_G)\}$, 
      any $\hat{u}\in\oball^d(0;1)\cap \operatorname{Im}(G(x_0))$, and any $k$ such that $|k|\leq \|G(x_0)^\dagger\|^{-1}-(L_f+L_G)|x|$. \Qed
 \end{thm}

Consequently, by \cite[Theorem 3]{Shafa23Reachability}, the GRS 
of \eqref{E: sys} can be calculated based on a proxy equation  of the  form 
\begin{small}
    \begin{equation}\label{E: proxy}
        \dot{\hxb}(t) = a + (b - c|\hxb(t)|) \hub(t), \;\;\hxb(0)=x_0, 
    \end{equation}
\end{small}

\noindent on the domain $\B$, where $a = f(x_0)$, $b:=\|G(x_0)^\dagger\|^{-1}$,  $c: = L_f+L_G$, and $\hub: [0, \infty)\ra\huu$ for  $\huu = \oball^d(0;1)\cap \operatorname{Im}(G(x_0))$. \ymmarkr{It can also be verified from (H1) and (H2) that $a\in \operatorname{Im}(G(x_0))$ \cite{Shafa23Reachability}.} 
For any $x_0\in\inte(\B)$, we denote by $\hphi_{\hub}(\cdot, x_0):[0, \infty) \ra\R^d$ the controlled flow map (solution) under $\hub$, by $\hPhi(x_0; \huu)$ the set of solutions. We also use  $\hphi_{\hub}(\cdot)$ to denote the solution when the initial condition is not emphasized. Let $T\geq 0$,  \ymmarkr{the reachable set at time $T$ and up to   $T$ (i.e., the GRS up to $T$) are denoted by $\hrr^T(x_0)$   and $\hrr^{\leq T}(x_0)$, respectively.  }

Then, for all $T\geq 0$ and all $x_0\in \B$, $\hrr^{\leq T}(x_0)\subseteq \rr^{\leq T}(x_0)$.
Eq.~\eqref{E: proxy} demonstrates the `slowest' growing rate for all $(f,G)\in\ddc$. It is worth noting that the  proxy system \eqref{E: proxy} is not an approximation of the true dynamics. \ymmarkr{For simplicity, we assume \( x_0 = 0 \). Otherwise, shifting coordinates via \( \tilde{x} = x - x_0 \) yields \( \dot{\tilde{x}} = \tilde{f}(\tilde{x}) + \tilde{G}(\tilde{x}) u \), with \( \tilde{f}(\tilde{x}) = f(\tilde{x} + x_0) \), \( \tilde{G}(\tilde{x}) = G(\tilde{x} + x_0) \), and \( \tilde{x}_0 = 0 \). The parameters \( a = \tilde{f}(0) \), \( b = \|\tilde{G}(0)^\dagger\|^{-1} \), and \( c \) in the proxy dynamics after the coordinate shift remain unchanged.
}



\subsection{Reachable Path of the Proxy System}\label{sec: reference_path}
In view of Theorem \ref{thm: track}, for any path generated by \eqref{E: proxy}, there always exists a control signal $\ub$ for \eqref{E: sys} that ensures the system to follow the same path. \ymmarkr{This motivates us to identify a valid reachable path generated by system \eqref{E: proxy}, which can approach a  small neighborhood of a target point on the boundary of GRS.} To do this, we  investigate the proxy system in this subsection \ymmarkr{beyond the properties stated in \cite{Shafa23Reachability}, specifically for the purpose of control synthesis}. Below, we present some facts  for  \eqref{E: proxy}. The proofs can be found in Appendix \ref{sec: proof_proxy_preliminaries}.

\ymmark{
\begin{prop}\label{lem: fact} 
We assume that there exists a $\Omega\subseteq \mathbb{B}$ such that $|a|< b-c\cdot\sup_{x\in\Omega}|x|$. Consider system \eqref{E: proxy} on $\Omega$. 
Let $T$ be such that $\hrr^{\leq T}(x_0)\subseteq\Omega$. Then,  for
   any $y\in \partial \hrr^{\leq T}(x_0)$, 
   $\not\exists\hphi_{\hat{\mathbf{u}}}(\cdot)\in \hPhi(x_0; \huu)$ such that $\hphi_{\hat{\mathbf{u}}}(t) = y$ for any $t\in[0, T)$. 
\Qed
\end{prop}

The above property states that $\hrr^{\leq T}(x_0)$ expands monotonically as $T$ increases for  relatively small $|a|$\footnote{For values of $a$ that do not satisfy the condition in Proposition \ref{lem: fact}, specifying a point on $\partial \hrr^{\leq T}(x_0)$ may result in the point being reached at some $t < T$ due to the strong drift. Additionally, the predicted GRS significantly deviates from the true reachable set, even over a short time horizon, as demonstrated in \cite[Section V.A]{Shafa23Reachability}. Therefore, it would be less practical to consider the same reachability control task as discussed in the remainder of this article based on the knowledge of the GRS. Nevertheless, we provide a reachability analysis in Appendix D.}.  We therefore intend to work with $\partial \hrr^{T}(x_0)$ rather than $\partial \hrr^{\leq T}(x_0)$. In terms of computing $\partial \hrr^{T}(x_0)$, although existing methods  
offer guaranteed precision \cite{rungger2018accurate, ramdani2011computing, han2006reachability}, they require long computation times and are   unsuitable for online learning and control synthesis. We thus introduce \begin{small}
    $\partial\hrr_{\text{pse}}^{T}(x_0) := \{ \hphi_{\hub}(T) : \hub \equiv \nu,\, \nu \in \partial\hat{\mathcal{U}} \}$
\end{small} and adopt a Monte Carlo method for its computation.

\begin{prop}\label{prop: reachable_path_0}
 Given 
 a $T$ such that $\hrr^{\leq T}( x_0)\subseteq\inte(\mathbb{B})$.   Then, for $a= 0$, $\partial\hrr_{\text{pse}}^{T}(x_0)= \partial\hrr^{T}(x_0)$.  Consequently, for any $y\in \partial\hrr^{T}(x_0)$,  
 it can be reached under $\hub\equiv\frac{y}{|y|}$. 
\Qed
\end{prop}}

\ymmark{ 
When $a\neq 0$, $\partial \hrr_{\text{pse}}^{T}(x_0)$ is generally not equal to $\partial \hrr^{T}(x_0)$\footnote{We kindly refer readers to Appendix C and Proposition C.1   for counterexample, which indicate that the control signal that maneuvers   $\hphi_{\hub}$ to the boundary of the GRS always has full magnitude, but its direction may vary over time.}. However, we have the following approximation. 

\begin{prop}\label{prop: approx}
  Consider $\dot{\bar{\xb}}(t) = (b - c|\bar{\xb}(t)|)\hub(t)$ and its solution $\bar{\phi}_{\hub}(t)$. Then, 
  it follows that $\vartheta_a(t):=|\hphi_{\hub}(t) - \bar{\phi}_{\hub}(t) - at| \leq \frac{e^{ct}-(1+ct)}{c}|a|$. \Qed 
\end{prop}}

\ymmark{

Now, we  denote $\bar{\rr}^T(x_0)$ as the reachable set for the driftless $\bar{\phi}_{\cdot}(\cdot)$ at $T$. 
As a direct consequence of Proposition~\ref{lem: fact}-\ref{prop: approx}, for any $y \in \partial \hat{\rr}^T(x_0)$, there exists a control $\hub$ such that $\hat{\phi}_{\hub}(T) = y$, and hence $y\in\{x\in\R^n: |x|_{\partial \bar{\rr}^T(x_0) + aT} \leq \vartheta_a(T)\}$. By the same logic, for any $y \in \partial \bar{\rr}^T(x_0)+aT$,  we have $y\in\{x\in\R^n: |x|_{\partial \hat{\rr}^T(x_0)}\leq \vartheta_a(T)\}$. 
 Therefore, 
 in the Hausdorff distance, $d_H(\partial\bar{\rr}^T(x_0)+aT-\partial\hat{\rr}^T(x_0))\leq\vartheta_a(T)$. Similarly, one can apply Proposition \ref{prop: reachable_path_0} and show that $d_H(\partial\bar{\rr}^T(x_0)+aT-\partial\hat{\rr}_{\text{pse}}^T(x_0))\leq\vartheta_a(T)$. Note that, for sufficiently small $T$,  \begin{small}
     $\vartheta_a(T)$ is approximately bounded by $ |a|\cdot(T^2/2+\mathcal{O}(T^3))$. 
 \end{small} 

Let
$\rho := \sup_{{x \in \hrr^T_{\text{pse}}(x_0)}} |x|$
and
$\Delta_a := \frac{|a|}{b - c\rho}$.
For the remainder of the paper, we focus on the scenario where \begin{small}$\partial\hat{\rr}_{\text{pse}}^T(x_0)\subseteq\mathbb{B}$ and $\Delta_a < 1$. \end{small}
Note that for any $y \in \partial\hat{\rr}_{\text{pse}}^T(x_0) \subseteq \operatorname{int}(\mathbb{B})$, as verified by the definition of $\partial\hat{\rr}_{\text{pse}}^T(x_0)$, one can use $\hub \equiv \frac{y - aT}{|y - aT|}$ in the proxy system to reach $y$. On the other hand, based on the above facts about approximations using the driftless proxy system (which enable easier computation), we can verify that the trajectory created by $\hub \equiv \frac{y - aT}{|y - aT|}$ deviates from the straight-line segment toward $y$ by no more than \begin{small}
$\frac{cT^2}{2}\sqrt{|a|^2 - \langle a, y \rangle^2 / |y|^2} + \mathcal{O}(T^3)$,
\end{small} provided that $T$ is sufficiently small relative to $|a|$. One can then directly use the straight-line segment toward $y$ as the reference path for control synthesis in the next section, enabling simpler analysis and computation.

}

\ymmark{
}


    \section{Control Synthesis}\label{sec: control}
In this section, our objective is to synthesize a controller that leads the trajectory to eventually reach a neighborhood of a specified \ymmark{$\partial\hat{\rr}_{\text{pse}}^T(x_0)$} for some $T>0$.  
We take advantage of the properties of \eqref{E: proxy} as stated in Section \ref{sec: proxy_preliminaries} and search for control inputs for the system~\eqref{E: sys} such that \( \phi_{\ub}(t)  \) follows the   reference path toward \( y   \) with the same velocity as \( \hphi_{\hub}(t) \).

    In practice, however, it is impossible to accurately learn the transition of \eqref{E: sys} based on a single system run, and hence achieving an exact path-following strategy is unattainable. Considering this situation, to leverage the connection between systems \eqref{E: sys} and \eqref{E: proxy}  and fully utilize the known information, we must necessarily relax the reachable time $T>0$; that is, we consider finite-time reachability rather than reachability at the exact time $T$. We  create a `correct-by-construction' reference path from \eqref{E: proxy}, based on up-to-date trajectory information.

Furthermore, due to the limited information and potential nonlinearity of $f$ and $G$,  long-term predictability of the trajectory is not achievable. 
The controllers must 
provide  infinitesimal  direction to constantly track the reference path, 
which will in turn force the controlled 
$\phi_{\ub}(t)$ to be maintained within a small neighborhood of the reference path. To better illustrate the idea,  
we 
focus on the case where $a=0$.

\subsection{Preliminaries for System Learning}\label{sec: learning_pre}

In order to learn $f(x)+G(x)u$ at points along any single-run trajectory, 
we consider a piece-wise constant control $\ub$, where each piece has a duration of $\dt$. To be more specific, the controller can apply any $m+ 1$ affinely independent
constant inputs for a short period of time $\dt$  \cite{ornik2019control}. Thus, the entire learn-control
cycle is of length $\tau=(m+1)\dt$. 

Let $\tau_n:=n\tau$ for $n\in\N$, and let $\{\tau_n\}_n$ be the sequence of instants at which to decide velocity. 
Consider the index sets $\I:=\set{0, 1, \cdots, m}$ and $\I_0:=\set{1, 2,  \cdots, m}$. We then denote $\{u_{n,j}\}_{j\in\I}$ by the affinely independent sequence of constant inputs, where  $u_{n,j}$ is applied within $[\tau_n+j\dt, \tau_n+(j+1)\dt\ymmarkr{)}$ for each $j\in\I$. In this case, each $u_{n,0}$ is determined at $\tau_n$ by a strategy aimed at returning a nearly optimal direction attracted to the reference path. Meanwhile, $\{u_{n,j}\}_{j\in\I_0}$ are subsequent  inputs that are selected according to a fixed procedure. 

We propose the following inductive procedures to achieve the  control task, with a detailed explanation of their feasibility.  
The proofs can be found in Appendix \ref{sec: proof_control}. 

\subsection{Learning of System Dynamics}\label{sec: learning_dyn}
Suppose that $u_{n, 0}$  has been determined at $\tau_n$ for each $n$. 
We then proceed to learn the system dynamics, denoted as $v_x(u):=f(x)+G(x)u$
at $\tau_{n+1}$ for each $n$. 
We follow the learning algorithm in the reference \cite[Section V]{ornik2019control}, which is based on the trajectory information within each learn-control cycle, spanning the interval $[\tau_n, \tau_{n+1})$. This is achieved by leveraging the control affine form of system \eqref{E: sys},  and   by employing an
affinely independent sequence of constant inputs 
\begin{equation}\label{E: u_nj}
    u_{n,j}:=u_{n,0} + \Delta u_j
\end{equation}
within $[\tau_n+j\dt, \tau_n+(j+1)\dt]$ for each $j\in\I_0$.  Here, $\Delta u_j:=\pm\epsilon \eb_j$, for  $\epsilon>0$ representing a small amplitude, 
and $\{\eb_j\}_{j\in\I_0}$ being the set of orthornormal unit vectors in $\R^m$. 

Denote \begin{small}
    $M_0:= \max\set{\sup_{x\in\B}|f(x)|, \sup_{x\in\B}|G(x)|}$,  \ymmarkr{$C_0:=M_0(m+1)$, $C_1:=M_0(m+1)^2$, $C_2:=2M_0\lmax(m+1)^2$,  and 
    $C_3:=M_0\lmax(m+1)^3$}. 
\end{small}  Let \(x_{n, j}=\phi_\ub(\tau_n+j\dt, x_0)\) and $\xk_{n+1}:=x_{n,m+1}$. Then, it is clear that $\xk_{n+1}=x_{n+1, 0}$.  We   have the following approximation precision \cite[Lemma 4]{ornik2019control}: 
\begin{enumerate}
    \item[(1)] $|\phi_\ub(t_1, x_0)-\phi_\ub(t_2, x_0)|\leq \ymmarkr{C_0|t_2-t_1|}$, $\forall t_1, t_2\in[\tau_n, \tau_{n+1}]$. In particular, $\ymmarkr{|\xk_{n+1}-\xk_n|\leq C_1\cdot\delta t}$; 
    \item[(2)] $|\frac{(x_{n,j+1}-x_{n,j})}{\dt}- v_{x_{n,j+1}}(u_{n,j})|\leq \ymmarkr{(C_2/4)\cdot\dt}$, $\forall j\in\I$. 
    \item[(3)] $|v_{x_{n,j+1}}(u_{n,j})-v_{\xk_{n+1}}(u_{n,j})|\leq \ymmarkr{C_3\cdot\dt}$,  $\forall j\in\I$.
\end{enumerate}

By introducing the following class of parameterized control inputs, we can approximate the system dynamics using only the information from the trajectory data and the parameters. 
\begin{deff}\label{def: lambda}
\ymmarkr{Let \begin{small}
    $\Lambda:=\{\lambda=\{\lambda_j\}_{j\in\I}\subseteq\R_{\geq 0}: \sum_{j\in\I}\lambda_j=1\}$.
\end{small}}
   A parameterized control input w.r.t. \ymmarkr{$\lambda\in\Lambda$} is of the form
\ymmarkr{$u_{n,\lambda} :=\sum_{j\in\I}\lambda_ju_{n, j}$}. 
   We  define the set of parameterized control inputs as \ymmarkr{$\ymmarkr{\uu_{n,\lambda}:=\{u_{n,\lambda}\in\uu:\; \lambda\in\Lambda\}}$}. \Qed     
\end{deff}

Note that, by \cite[Theorem 5]{ornik2019control}, we have the following bound 
\begin{small}
    $\left|v_{\xk_{n+1}}(u_\lambda)-\sum_{j\in\I}\lambda_j(x_{n,j+1}-x_{n,j})/\dt\right|\leq \ymmarkr{\ee_\lambda(\dt, \epsilon)}$
\end{small} for $u_{n,\lambda}\in\uu_{n,\lambda}$ and 
    \ymmarkr{$\ee_\lambda(\dt, \epsilon)=2((4m^{\frac{3}{2}}+\epsilon)/\epsilon) \cdot C_3\cdot \dt$}.
    This bound,  which depends on two parameters,  indicates the necessity of considering their joint effect to converge to $0$, rather than adjusting each parameter individually.
We   then use the approximation $\tv_{\xk_{n+1}}(u_{n,\lambda}):=\sum_{j\in\I}\lambda_j(x_{n,j+1}-x_{n,j})/\dt$
of $v_{\xk_{n+1}}(u_{n,\lambda})$ for the decision-making at $\tau_{n+1}$. 

\subsection{Control Design at $\tau_n$}\label{sec: control_at_tau_n}

\subsubsection{Initial Control Input}

For $n=0$, the control input $u_{0,0}$ is considered to initialize the system \eqref{E: sys}, such that  the direction of unknown system within $[0, \dt]$ is expected to closely follow that of the proxy system  \eqref{E: proxy}. As $f(x_0)$ and $G(x_0)$ are known, we use the following argument to determine $u_{0, 0}$. 
Since $\dt >0$ is arbitrarily small, we have  $\phi_{\ub}(\delta t)    =\delta t\cdot (f(x_0)  + G(x_0)u_{0,0}) + \mathcal{O}(\delta t) = \delta t\cdot  (G(x_0)u_{0,0}) + \mathcal{O}(\delta t^2)$, 
where $\mathcal{O}(\dt^2)$ is a high order term w.r.t. $\dt$. 
On the other hand, in view of Proposition \ref{prop: reachable_path_0}, the control $\hub$ leading to the optimal direction for the proxy system \eqref{E: proxy} is the constant signal $\hub\equiv\frac{y}{|y|}$. Therefore, $\hphi_{\hub}(\delta t)    =   \frac{y}{|y|}\cdot \int_0^{\delta t} (b-c|\hphi_{\hub}(s)|) \;ds 
             =  (1-\epsilon)\frac{b\cdot y\cdot \delta t}{|y|}+\epsilon\frac{b\cdot y\cdot \delta t}{|y|}+ \mathcal{O}(\delta t^2)$, 
             recalling $b:=\|G(x_0)^\dagger\|^{-1}$ and $c: = L_f+L_G$. 
The parameter $\epsilon>0$, representing a small amplitude used in Equation \eqref{E: u_nj}, is reserved for future attempts to employ multiple small-wiggling, independent inputs $\{\Delta u_j\}$ for learning the dynamics, as discussed in Section \ref{sec: learning_dyn}.
By equating the first-order terms, $\dt\cdot(G(x_0)u_{0,0})=(1-\epsilon) b\cdot y\cdot \dt/|y|$, we can   set the  initial constant control $u_{0,0} = (1-\epsilon) \frac{G(x_0)^\dagger (y) }{\|G(x_0)^\dagger\||y|}$
for the system \eqref{E: sys}.  
   Recalling \eqref{E: u_nj}, the term $1-\epsilon$ is used to ensure that the sequence   $\set{u_{0, j}}$ is  a subset of $\uu$.

\subsubsection{Control Design at Other $\tau_n$}
Recall notations in \ref{sec: learning_dyn}. 
We aim to create a sequence   $\set{z_n}$, such that $z_n=\theta_ny$ and $\set{\theta_n}$ is an increasing sequence of real numbers within $[0, 1]$ converging to $1$. Then, within each learning cycle, we design control inputs to ensure that $\phi_\ub(t)$ approaches each $z_n$. In order to dynamically measure the distance of $\phi_\ub(t)$ with the reference path, we now introduce the function $d_z(x) = |x - z|^2$ for $z = \theta y$ and $\theta\in[0, 1]$. 

To guarantee feasibility, we need to demonstrate that the sequence $\set{z_n}$, which satisfies the previously mentioned property,  can be constructed based on the  trajectory data. Furthermore, for any $t\in(\tau_n, \tau_{n+1}]$, there should exist $\set{u_{n,0}}$ (along with $\set{u_{n,j}}$ introduced in Section \ref{sec: learning_dyn}) that also satisfies 
\begin{small}
    \begin{equation}\label{E: lie_derivative}
\begin{split}
            \dot{d}_{z_n}(\phi_\ub(t))=\langle\nabla d_{z_n}(\phi_\ub(t)),\; v_{\phi_\ub(t)}(\ub(t))\rangle<0. 
\end{split}
    \end{equation}
\end{small}

We first argue inductively to show that the above properties can be satisfied. We also refer to  Figure \ref{fig: lem} for a visualization. 

\begin{lem}\label{lem: z_n_geq_1}
    Let $n\geq 1$ and $r>0$ be fixed. Let $\ub(t)=u_{n,j}$ for $t\in[\tau_n+j\dt, \tau_n+(j+1)\dt)$ and for $j\in\I$. Let  $\xk_n$ be given and consider $z_n=\theta_ny$ with $|z_n-\xk_n|=r$ for some $\theta_n\in(0, 1)$. Suppose that $\dot{d}_{z_n}(\phi_\ub(t))<0$ for all $t\in[\tau_n, \tau_{n+1})$. Then, there exists a  $z_{n+1}=\theta_{n+1}y$ such that $\theta_{n+1}>\theta_n$ and $|z_{n+1}-\xk_{n+1}|=r$. Particularly, $r-|\xk_{n+1}-z_n|\leq |z_{n+1}-z_n|< 2r$.\Qed
\end{lem}

\begin{figure}[!t]
\centerline{\includegraphics[scale = 0.34
]{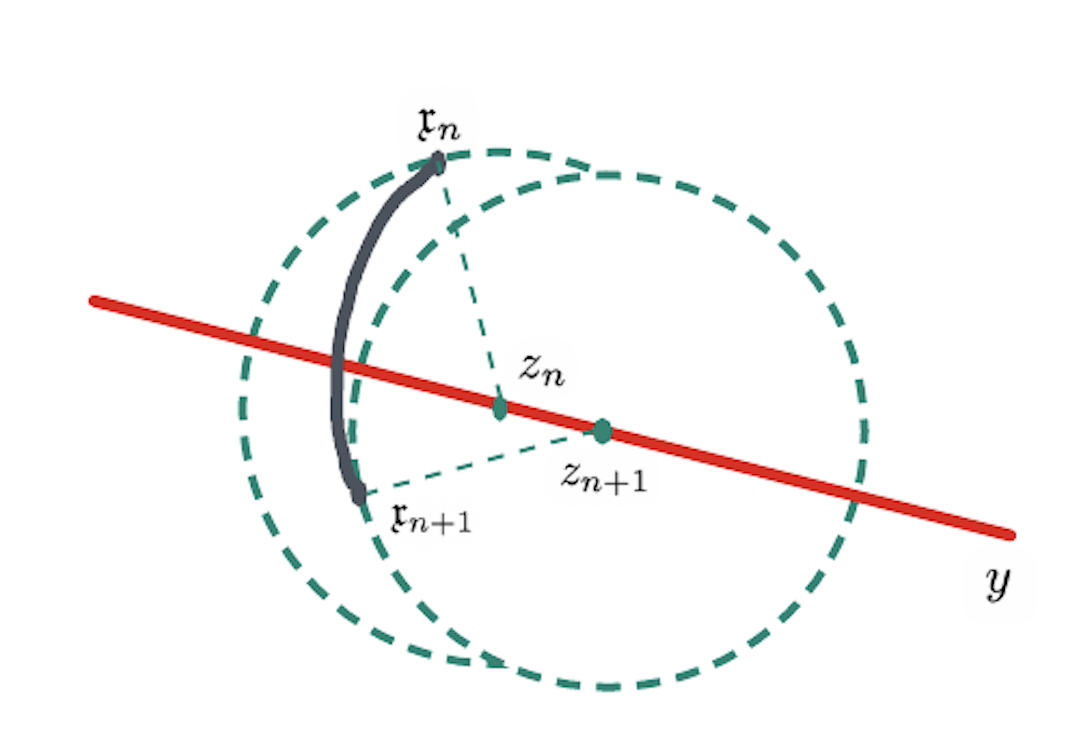}
}
\caption{An illustration of Lemma \ref{lem: z_n_geq_1}.  The point $z_n$ is given  as the point on the line segment between $0$ and $y$ with $|z_n-\xk_n|=r$. Suppose  $\dot{d}_{z_n}(\phi_\ub(t))<0$ for all $t\in[\tau_n, \tau_{n+1})$, then $|\xk_{n+1}-z_n|<r$, and $z_{n+1}$ can be found on the line segment with $|z_{n+1}-\xk_{n+1}|=r$ and moves closer to $y$. }
\label{fig: lem}
\end{figure}

It then suffices to show that there exists a determination strategy for $r$, $\epsilon$, and $\dt$, such that the above construction of $\set{z_n}$ can be initialized, and there exists $\set{u_{n,0}}$ 
such that $\dot{d}_{z_n}(\phi_\ub(t))<0$ for all $t\in[\tau_n,\tau_{n+1})$.  In selecting $r$ to incorporate the learning process described in Section \ref{sec: learning_dyn}, we consider the quantity $r:=r(k, \delta t) = 
\sup_{\{\hub\equiv\hat{u}\in\partial\huu\}}|\hphi_{\hub}(k\tau, x_0)-x_0|$. The following lemma \ymmarkr{provides a set of conditions on $k$, $\epsilon$, and $\delta t$ under which the sequence ${z_n}$ can be initialized; these conditions can always be satisfied through appropriate tuning, particularly with sufficiently small $\epsilon$ and $\delta t$.}

\begin{lem}\label{lem: r_initial}
    Let  $\ub(t)=u_{0,j}$ for $t\in[j\dt, (j+1)\dt)$ and for $j\in\I$. 
    Consider $z_0=x_0=0$. Suppose $\epsilon$ and $\dt$ 
    satisfy   \ymmark{$\epsilon>(C_3/C)\delta t $. Recall   \begin{small}
    $\rho:=\sup_{\{x\in\hrr^T_{\text{pse}}(x_0)\}}|x|$.
\end{small} 
If we also have $k>Cb/(b-c\rho)$ and $2kb\tau/(b-c\rho)<T$}, 
    then, there exists a  $z_{1}=\theta_{1}y$ such that $\theta_{1}\in(0, 1)$ and $|z_{1}-\xk_{1}|=r:=r(k, \dt)$.\Qed
\end{lem}

We now  show the existence of $u_{n,0}$ such that \eqref{E: lie_derivative} holds  
for all $n$. The following statement verifies the existence of the control input for \eqref{E: sys} such that \eqref{E: lie_derivative} is satisfied at $\tau_n$.

\begin{prop}\label{prop: sign_tau_n}
For each $n\geq 1$, let $\xk_n$ be
given and $z_n$ be determined by Lemma \ref{lem: z_n_geq_1} with $r:=r(k, \dt)$.
    Then, there exist a $u\in\uu$ such that $\langle\nabla d_{z_n}(\xk_n),\; f(\xk_n)+G(\xk_n)u\rangle \leq -2r(b-c|\xk_n|).$ Particularly, we can determine $u$ as   $\operatorname{argmin}_{u\in\uu}\langle\nabla d_{z_n}(\xk_n),\; f(\xk_n)+G(\xk_n)u\rangle$. \Qed
\end{prop}

We then continue to show that, within each learning cycle $[\tau_n, \tau_{n+1})$, by considering the subsequential small-wiggling inputs $\{\Delta u_j\}$ as described in Section \ref{sec: learning_dyn}, we have $\dot{d}_{z_n}(\phi_\ub(t))<0$ for all $t\in[\tau_n, \tau_{n+1})$ for small $\dt$ and $\epsilon$. 

\begin{thm}\label{thm: main_direction}
Recall notation $b$ and $c$ in \eqref{E: proxy}.  \ymmarkr{Define $\ee_r(\dt, \epsilon):=C_2\dt+\epsilon M_0$ and $\ee_n(\dt, \epsilon):=2M_0C_1\dt+C_1\ee_r(\dt, \epsilon)$.   Suppose that $  \ee_n(\dt, \epsilon)\leq r(b-c|\xk_n|-\ee_r(\dt, \epsilon)) $}. Let $\ub(\tau_n)=u$ be the control input as in Proposition \ref{prop: sign_tau_n}. Let $\ub(t)=u_{n,j}$ for all $t\in[\tau_n+j\dt, \tau_n+(j+1)\dt)$ and for each $j\in\I$. 
 Then $\dot{d}_{z_n}(\phi_\ub(t))<0$ for all $t\in[\tau_n, \tau_{n+1})$. \Qed
\end{thm}

\begin{rem}
For $t\in[\tau_n, \tau_{n+1})$, 
the deviation  $|\phi_\ub(t)-\xk_n|$ should be small 
as presented in  \cite[Lemma 4]{ornik2019control}. 
Consequently, the uncertain deviations $|f(\phi_\ub(t)) - f(\xk_n)|$ and $|G(\phi_\ub(t))-G(\xk_n)|$ are small. The condition \ymmarkr{on the small terms $\ee_n$ and $\ee_r$}  ensures that \ymmarkr{the attracting force $-2r(b - c|\xk_n|)$ at each $\tau_n$ (as in Proposition \ref{prop: sign_tau_n}) continues to dominate over the interval $(\tau_n, \tau_{n+1})$}, so that the flow remains attracted to $z_n$ even in the presence of such uncertainty. 

  Recalling that $r=r(k, \dt)$ represents  the maximal length of $\phi_\ub(k\tau, x_0)$ under constant inputs for   \eqref{E: proxy}, to satisfy the stated condition, one needs consider the joint effect of   $k$ and $\dt$. 
    Note that for $\dt$ sufficiently small, $r\approx kb\tau$.  
    For small $b$ and $k=1$,   \ymmarkr{the main contributor  $2M_0C_1\dt\leq r(b-c|\xk_n|)$ in the condition stated in Theorem \ref{thm: main_direction} may not hold.} 
    However, this is not problem in view of  the proof of Lemma \ref{lem: r_initial}. We consider the \ymmarkr{joint effect of tuning $r$ by increasing $k$ and reducing $\dt$, ensuring that $r$ increases but not to an arbitrarily large value, so as to preserve reachability precision}. \Qed
\end{rem}

So far, we have demonstrated that the control input can be determined at $\tau_n$ to satisfy the requirement in Lemma \ref{lem: z_n_geq_1} by searching over $\uu$. 
However, the dynamic learning procedure,  as demonstrated in Section \ref{sec: learning_dyn},  only permits the use of a subset of $\uu_{n,\lambda}$ and achieves an approximation of the velocity. Therefore, we need to take the approximation error into account. \ymmarkr{Recall Definition \ref{def: lambda} and $\tv_{\xk_{n+1}}(u_{n,\lambda})$ in Section \ref{sec: learning_dyn}. Define $\dot{\tilde{d}}_{z_n}(n, \lambda):=2\langle \xk_n-z_n, \tv_{\xk_{n+1}}(u_{n,\lambda})\rangle$.} We then look at the following slight modification of \cite[Theorem 9]{ornik2019control}.
\begin{thm}\label{thm: suboptimal}
   For a fixed $n\geq 1$, we have \begin{small}
        \ymmarkr{$\left|\min\limits_{\lambda\in\Lambda}\dot{\tilde{d}}_{z_n}(n, \lambda)-\min\limits_{\ub(\tau_n)\in\uu}\dot{d}_{z_n}(\phi_\ub(\tau_n))\right|\leq 2\ee_\mu(\dt, \epsilon),$}
    \end{small}
    where \begin{small}
        $\ymmarkr{\ee_\mu(\dt,\epsilon)= (3L_d(C_4/C_3)  \ee_\lambda(\dt, \epsilon)
            +L_dC_0\dt)/2}$;
    \end{small}
 \ymmarkr{$L_d$} is the Lipschitz constant of $d_{z_n}$; \begin{small}
     $\ymmarkr{C_4=(M_0+1)(L_\text{max}+1)(m+1)^3}$.
 \end{small} \Qed
\end{thm}

Considering the sub-optimality using \ymmarkr{$\lambda\in\Lambda$ and hence $\uu_{n,\lambda}$}, we have the following guarantee. We omit the proof due to the similarities to the proof of Theorem \ref{thm: main_direction}. 
\begin{cor}\label{cor: decision}
\ymmarkr{Recall the error terms $C_r$ and $C_n$ in Theorem \ref{thm: main_direction}. 
Suppose that $\ee_n(\dt, \epsilon)+\ee_\mu(\dt, \epsilon)\leq r(b-c|\xk_n|-\ee_r(\dt, \epsilon)+\epsilon M_0)$. 
Let $\lambda^\star = \operatorname{argmin}_{\lambda\in\Lambda}\dot{\tilde{d}}_{z_n}(n, \lambda)$ and $\ub(\tau_n)=u_{n,0}:=(1-\epsilon)\sum_{j\in\I}\lambda^\star_ju_{n-1, j}$}  be the controller at $\tau_n$.
 Then $\dot{d}_{z_n}(\phi_\ub(t))<0$ for all $t\in[\tau_n, \tau_{n+1})$. \Qed
\end{cor}

\subsection{Summary of Algorithm}\label{sec: summary}
Recall notation $a$, $b$ and $c$ in \eqref{E: proxy}, \ymmarkr{as well as $\rho$ and $\Delta_a$ in Section \ref{sec: reference_path}. For $a \neq 0$ but with a small norm relative to $b$ (or small $\Delta_a$) as illustrated by \cite[Section V.C]{Shafa23Reachability}}, the GRS prediction maintains reasonable accuracy over a longer time horizon. \ymmark{In this scenario, the control synthesis technique and principle stated in Section~\ref{sec: control_at_tau_n} remain the same but require the following slightly modified sufficient conditions for determining the parameters. (i) In selecting $r = r(k, \delta t)$, we follow Lemma~\ref{lem: r_initial} under the conditions
\begin{small}
\begin{equation}\label{E: cond_2}
\epsilon > (C_3 / C)\delta t; \quad \frac{2kb\tau}{b - c\rho} < T; \quad k > \frac{Cb}{b - c\rho} + \Delta_a,
\end{equation}
\end{small}

\noindent which are the same as in the original case except for the last one.} (ii) To maintain $\dot{d}_{z_n}(\phi_{\ub}(t)) < 0$ for all $t \in [\tau_n, \tau_{n+1})$ for each $n\geq 1$, we follow Corollary~\ref{cor: decision} and choose the same controller as stated therein, but under a modified condition 
\begin{small}
    \begin{equation}\label{E: cond}
\begin{split}
\ymmarkr{\ee_n(\dt, \epsilon)+\ee_\mu(\dt, \epsilon)\leq r(b-c\rho-|a|-\ee_r(\dt, \epsilon)+\epsilon M_0)}.
\end{split}
\end{equation}
\end{small}

\noindent The proofs follow exactly the same procedure as the corresponding statements we modify from, and is therefore omitted due to repetition. 
The condition suggests that the value of $a$ does not alter the direction of attraction.
We summarize the algorithm for in Algorithm \ref{alg: one_time}. 

\begin{prop}\label{prop: accuracy}
    Following Algorithm \ref{alg: one_time}, 
    $\phi_\ub$ will eventually reach $\oball^d(y; 2r)$. \Qed
\end{prop}

\setcounter{algorithm}{0}
\begin{algorithm}[H]\label{alg: alg}
	\caption{Control Synthesis}\label{alg: one_time} 
	\begin{algorithmic}[1]
      \Require $d$, $m$,  $x_{0, 0}:=x_0$, $T$, \ymmark{$y\in\partial \hat{\rr}^T_{\text{pse}}(x_0)$}, and $u_{0, 0}= (1-\epsilon) \frac{G(x_0)^\dagger (y -x_0) }{\|G(x_0)^\dagger\||y -x_0|}$. 
		\Require  $x_0=0$; and $\dt$, $\epsilon$, $k$  based on the conditions in \ymmark{\eqref{E: cond_2} and \eqref{E: cond}}, where $r = \sup_{\{\hub\equiv \hat{u}\in\partial\huu\}}|\phi_\ub(k(m+1)\dt, x_0)|$

  \State $n=0$.
  \Repeat 
  \State $\tau_n = n(m+1)\dt$.
  \For{$j$ \textbf{from} 0 \textbf{to} $m$}
 \State  $\ub(t)\equiv u_{n,j}$ for all $t\in[\tau_n + j\dt, \tau_n+(j+1)\dt)$ based on Eq.~\eqref{E: u_nj};
  \State $x_{n,j+1}=\phi_\ub(\dt, x_{n,j})$.
  \EndFor
 \State  $\xk_{n+1}=x_{n,m+1}$.
 \State Determine $z_{n+1}$ based on Lemma \ref{lem: z_n_geq_1}. 
 \State Let 
      \ymmark{$u_{n+1, 0}=(1-\epsilon)\operatorname{argmin}_{\lambda\in\Lambda}\langle 2(\xk_{n+1}-z_{n+1}), \sum_{j\in\I}\lambda_j(x_{n,j+1}-x_{n,j}) \rangle,$} 
 where $\sum_{j\in\I}\lambda_j = 1$ and $\Lambda$ are defined in Definition \ref{def: lambda}. 
 \State $n:=n+1$.
 \Until $|z_n-y|<r$. 
	\end{algorithmic}
\end{algorithm}

    \ymmarkr{Note that \eqref{E: cond_2} and  \eqref{E: cond} provide theoretical guidance for parameter tuning to make Algorithm \ref{alg: alg} work. All perturbation terms in \eqref{E: cond} depend only on $\delta t$, $\epsilon$, and $k$, and are either linear in $\epsilon$ and $\delta t$, or arise from $\mathcal{E}_\lambda$ \cite[Theorem 5]{ornik2019control}, which originates from the conservative estimation in \cite[Theorem 4]{ornik2019control}. These terms can be made arbitrarily small under the restrictions of \eqref{E: cond_2}. Although the theoretical effectiveness of Algorithm \ref{alg: alg} can be guaranteed over a short time span for each learning cycle, \cite[Theorems 4–5]{ornik2019control} can be relaxed to reduce the restrictiveness of parameter tuning and better accommodate practical use, which is of interest for future research.}

\section{Case Study}\label{sec: case}
We use a control system with decoupled quadrotor dynamics to illustrate the proposed control synthesis algorithm, building on the same example   discussed in    \cite[Section V.C]{Shafa23Reachability}.
We examine the scenario where a UAV collides with an obstacle, leading to undesired velocity rotations \cite{chowdhary2013guidance, jourdan2010enhancing}.
To support more complex tasks, such as safe landing, it is essential to determine a reachable set of pitch and roll velocities, denoted as $p$ and $q$ respectively,  whilst synthesizing control inputs that lead to specified pitch and roll velocities without prior knowledge of the system's dynamics. 

We focus on the situation where the inertia in the $\xb$- and $\yb$-axes are identical, allowing the yaw rate to be directly altered by increasing the corresponding torque action without impacting $p$ and $q$. To simplify the problem, we trivially reduce the yaw state to be
 constant $\pi/2$.

 \begin{figure}[!t]
\centerline{\includegraphics[scale = 0.30
]{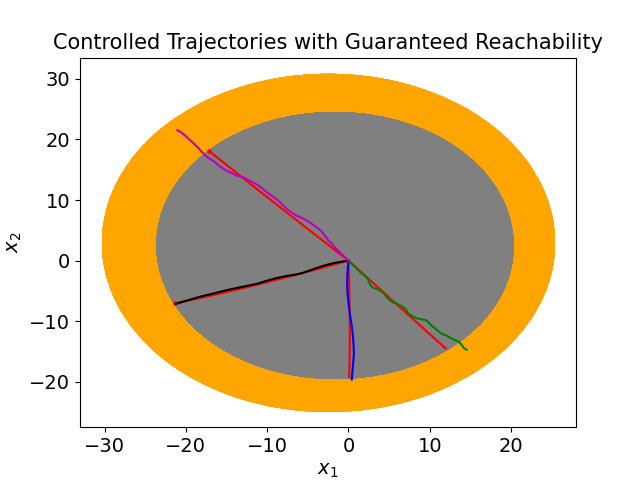}\includegraphics[scale = 0.30
]{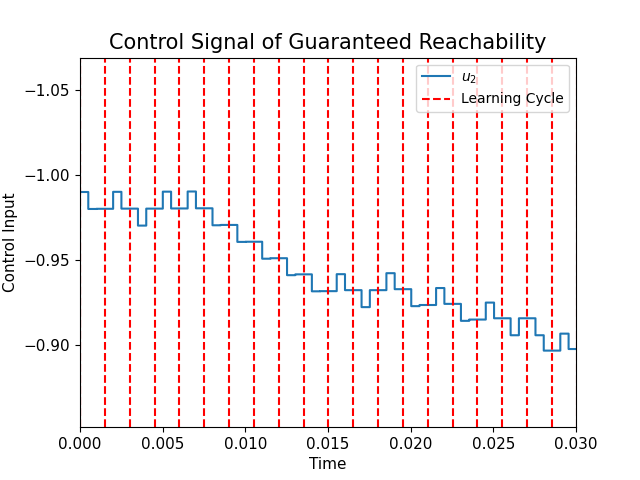}
}
\caption{(l.h.s.) Controlled trajectories for Scenario A-D. The grey region represents the GRS; 
the orange region indicates the true reachable set; the controlled trajectories for A-D, respectively, are shown in black, blue, green, and purple; the red trajectories represent the proxy controlled paths to four randomly sampled points on the boundary of the GRS for the true system aims to track. (r.h.s.) Selected partial control signals for Scenario B.}
\label{fig: path2}
\end{figure}

To apply the proposed algorithm, we let  the initial conditions
after collision be $(p_0, q_0)=(15, 10)$ radians per second, and then perform a coordinate transformation by introducing
$x_1=p-15$ and $x_2=q-10$.  The recast system is as follows.

\begin{small}
    \begin{equation}\label{E:sys2}
    \dfrac{d}{dt}\begin{bmatrix}
x_1(t)\\ x_2(t) 
\end{bmatrix}=\begin{bmatrix}
\frac{\pi(J_\yb-J_\zb)}{2J_\xb}(x_2(t)+10)\\
\frac{\pi(J_\zb-J_\xb)}{2J_\yb}(x_1(t)+15)
\end{bmatrix}+\begin{bmatrix}
\frac{1}{J_\xb}\tau_\phi\\ \frac{1}{J_\yb}\tau_\theta
\end{bmatrix},
\end{equation}
\end{small}

\noindent with initial conditions $x(0):=(x_1(0), x_2(0))=(0, 0)$ and 
torque inputs $(u_1, u_2):=(\tau_\phi, \tau_\theta) \in \oball^2(0; 1)$  for the roll and pitch. 
We also set 
$J_\xb = \frac{2MR^2}{5}+2l^2m, \;J_\yb=J_\xb, \;J_\zb=\frac{2MR^2}{5}+4l^2\text{m},$
where  $M=1\text{kg}$ and  $R=0.1\text{m}$ represent the mass and radius of the central frame. The central frame is connected to four point masses $\text{m}=0.1\text{kg}$, each representing one of the four propellers, positioned at an equidistant length of $l=0.5\text{m}$ from the central sphere.
Consequently, 
$a
\approx (-8.73, 13.09)^\top$,
and $b=\|G(0)^\dagger\|^{-1}\approx111.11$.  We also derive a conservative 
Lipschitz bounds to be $L_f=L_G =1$,  and therefore, $L_0=1$, and $\B=\oball^2(f(0); 111.11-2|x|)$ for the domain where \eqref{E: proxy} is well defined. 

\begin{rem}
    Note 
    that while the true system dynamics are described in \eqref{E:sys2}, they are unknown to the controller. The controller only has access to the trajectory data and the parameters required in Algorithm \ref{alg: one_time}. Hence, the choice of model does not affect the proposed control technique. Without real data, we can only leverage numerical ODE solvers based on the exact model  \eqref{E:sys2}  to generate trajectory data. 
    One can also randomly choose the aforementioned coefficients in \eqref{E:sys2}; regardless,  we only need an estimation of the parameters as mentioned in Section \ref{sec: proxy} for  GRS evaluation and subsequent control synthesis.  
    We would also like to refer readers interested in other simulations to \href{https://github.com/TahaShafa/reachable\_predictive\_control}{{https://github.com/TahaShafa/reachable\_predictive\_control}}, which includes video visualizations.  \Qed
\end{rem}

In the simulation, we aim to reach a neighborhood of any point on the boundary of the underestimated GRS for $T = 0.25$. 
To succinctly represent and  demonstrate the effectiveness of the proposed algorithm, we test it across four scenarios with different settings of the parameters $(\dt, \epsilon, k)$ and display the controlled paths in a single picture. Since the theoretical result on the guaranteed accuracy does not distinguish between target points, we randomly sample points on the boundary of the GRS. Detailed settings for $(\dt, \epsilon, k)$ are A: $ (0.0001, 0.005, 5)$; B: $(0.0005, 0.01, 6)$; C: $ (0.0008, 0.08, 12)$; 
D: $(0.0015, 0.10, 40)$. The theoretical accuracy, indicated by  $r$, can be computed based on the chosen parameters and is equal to $0.18, 1.11, 3.48, 18.83$ for the four cases, respectively. 

The controlled trajectories for A-D are presented  as shown in Figure \ref{fig: path2}. We choose to present only a zoomed-in view of the control signal for Scenario B to demonstrate the piecewise constant shape resulting from a short period of the learning cycle. 
The accuracy across scenarios A-D decreases as predicted. Specifically, Scenario A closely adheres to the trackable path, while  D significantly deviates from the intended trackable path.   As depicted in the r.h.s. picture of Figure \ref{fig: path2}, there are three sub-intervals within each learning cycle $n$, corresponding to $u_{n,j}$ for $j=0,1,2$. 

\section{Conclusion}
In this paper, we investigate the control 
properties of an underapproximated control system, whose reachable set represents a guaranteed reachable set of the true but unknown system. Then, we utilize the connection between this proxy system and the true system to synthesize controllers that guide the  trajectory to reach a neighborhood of a specified point within the guaranteed reachable set. 
The essence of the proposed method is to automatically generate a sequence of points $\set{z_n}$, converging to the required reachable point $y$, for the true trajectory to follow. This approach allows the use of historical data within a learning cycle to determine the trajectory's direction for the subsequent learning cycle. 

We also explore the sufficient conditions    for the controlled trajectory to approach each $z_n$ within every learning cycle, ensuring it will eventually reach a small neighborhood of the specified point $y$. It would be of significant engineering interest to \ymmarkr{improve the error estimation} and tune the parameters, as outlined in Algorithm 1, to automate the learning and control of an unknown system based on a single trajectory run.

For future work, building on the connections between \cite{Shafa23Reachability} and the recent work \cite{shafa2024guaranteed} on extending GRS evaluation to Riemannian manifolds, the proposed method in this study provides a foundation for controlling unknown systems under more relaxed assumptions, with applications to underactuated systems. 
Additionally, \ymmarkr{exploring how to control an unknown system to reach a region beyond the current short-time-frame estimation of the GRS is crucial}. This may involve online waypoint generation and adaptive control guided by iterative application of our current methods. 
\ymmarkr{As data accumulates through this process, recent advances in Koopman-based system identification \cite{meng2024resolvent, zeng2024data} can be leveraged for more precise dynamics learning. Robust control methods can then be employed to address the remaining control tasks with better performance, ultimately achieving resilience and guaranteed task completion for  unknown nonlinear control systems. }

\bibliographystyle{ieeetr}        
\bibliography{TAC}

\appendices

\section{Proofs in Section \ref{sec: proxy_preliminaries}}\label{sec: proof_proxy_preliminaries}

\textbf{Proof of Proposition \ref{lem: fact}:}
 We assume the contrary, as specified in the statement, then, there exists a $\hub$  and a $t\in[0, T)$ such that $\hphi_{\hub}(t)=y\in\partial \hrr^{\leq T}(x_0)$.  
   Then, 
    $\hphi_{\hub}(T) = y + a(T-t) + \int_t^T (b-c|\hphi_{\hub}(s)|)\hub(s) \;ds.$
Since $\huu$ is 
compact, convex, and centered at $0$, for any $s\in [t, T)$, \ymmark{we can find a $\hub(s)$ such that $\hub(s) = \frac{k(s) (y-x_0)-a}{b-c|\hphi_{\hub}(s)|}$, 
where $k(s)$ is some positive number. By assumption, it can be verified that  $\hub(s)\in\huu$ for all $s\in[t, T)$. Therefore, $\hphi_{\hub}(T)     = y + a(T-t) + (y- x_0)  \int_t^{T}k(s)ds - \int_t^T a ds  = y + (y- x_0)  \int_t^{T}k(s)ds.$
Let $\lambda: = 1 + \int_t^{\tau}k(s)>1$, then $\hphi_{\hub}(T)  = \lambda (y ) + (1-\lambda) x_0$. Since $y \in\partial \hrr^{\leq T}(x_0)$ and $x_0\in\inte(\hrr^{\leq T}(x_0))$, the point $\hphi_{\hub}(T) $ is located on the line extending from $x_0$ through $y$, and continues beyond $y$, in the direction from $x_0$ to $y$. Due to the convexity property of $\hrr^{\leq T}(x_0)$ as stated in \cite[Section III]{Shafa23Reachability}, and the above construction of the signal used to extend the trajectory outward from $x_0$ to $y$, we have that $\hphi_{\hub}(t) \notin \hrr^{\leq T}(x_0)$, which violates the assumption.}   We hence prove the statement. \pfbox

\noindent\textbf{Proof of Proposition \ref{prop: reachable_path_0}:}  \ymmark{Consider any $y\in\partial\hrr^T(x_0)$ and $\nu=\frac{y}{|y|}\in\partial \hat{\uu}$. For any $\hub$, define $\rho(t) = \langle \hphi_{\hub} (t), \nu\rangle$. Then, $\dot{\rho}(t) = (b-c|\hphi_{\hub} (t)|)\langle \hub(t), \nu\rangle$. Since $\rho(t) \leq |\hphi_{\hub} (t)|$ and $\langle \hub(t), \nu\rangle\leq |\hub(t)||\nu|=1$, we have $\dot{\rho}(t) \leq b-c\rho(t)$. On the other hand, let $\hub^*\equiv\nu$. Then, $\hphi_{\hub^*}(t)=\nu\int_0^t(b-c|\hphi_{\hub^*}(s)|)ds$ is in the same direction as $\nu$. Consider $\rho^*(t) = \langle \hphi_{\hub^*} (t), \nu\rangle$,  then  $\dot{\rho}^* (t) = (b-c|\hphi_{\hub^*}(t)|)\langle \nu, \nu\rangle = b- c\rho^*(t)$. By the \textit{Comparison Principle} \cite[Lemma 3.4]{khalil2002nonlinear}, we have  \( \langle \hphi_{\hub} (T), \nu \rangle \leq \langle \hphi_{\hub^*} (T), \nu \rangle \), which indicates that \( \hphi_{\hub^*} (T) \in \partial \hrr^T(x_0) \) according to the supporting hyperplane theorem \cite{boyd2004convex}. Additionally, by \cite[Corollary 3]{Shafa23Reachability}, $\hphi_{\hub^*} (T)$ can only be \( y \).
 The above indicates that $\partial\hrr^T(x_0)\subseteq\partial\hrr^T_{\operatorname{pse}}(x_0)$. By the fact that $\hrr^T(x_0)$ is convex and $\partial\hrr^T_{\operatorname{pse}}(x_0) \subseteq \hrr^T(x_0)$, we obtain $\partial\hrr^T_{\operatorname{pse}}(x_0) \subseteq \partial\hrr^T(x_0)$ \cite{tyrrell1970convex}, which proves the first part of the conclusion.  The second part follows immediately. \pfbox}

\noindent\textbf{Proof of Proposition \ref{prop: approx}}: \ymmark{Let $\eta(t) = \hphi_{\hub}(t)-at$, then $\eta$ solves $\dot{\eta}(t)=(b-c|\eta(t)+at|)\hub(t)$. We have $\vartheta_a(t) :=\eta(t)-\bar{\phi}_{\hub}(t)$.
Then, $\dot{\vartheta}_a(t) = (b-c|\eta(t)+at|)\hub(t)-(b-c|\eta(t)|)\hub(t)+(b-c|\eta(t)|)\hub(t)-(b-c|\bar{\phi}_{\hub}(t)|)\hub(t)$. Consequently, for $\vartheta_a(t)\neq 0$, 
\begin{equation*}
    \begin{split}
        \frac{d}{dt}|\vartheta_a(t)|  = & \vartheta_a(t)\cdot \dot{\vartheta}_a(t)/|\vartheta_a(t)|\leq |\dot{\vartheta}_a(t)|\\
         \leq &c||\eta(t)+at|-|\eta(t)||   + c||\eta(t)|-|\bar{\phi}_{\hub}(t)|| \\
         \leq & c|a|t + c|\vartheta_a(t)|,
    \end{split}
\end{equation*}
and \begin{small}
    $d(e^{-ct}|\vartheta_a(t)|)/dt \leq -ce^{-ct}|\vartheta_a(t)|+c|a|te^{-ct}+ce^{-ct}|\vartheta_a(t)|\leq c|a|te^{-ct}$.
\end{small} 
 Then, \begin{small}
     $e^{-ct}|\vartheta_a(t)|\leq c|a|\int_0^tse^{-cs}ds$,
 \end{small} and the conclusion follows.} 
 \pfbox

\section{Proofs in Section \ref{sec: control}}\label{sec: proof_control}

\noindent\textbf{Proof of Lemma \ref{lem: z_n_geq_1}:}
    Note that $\xk_{n+1}=\phi_{\ub}(\tau_{n+1})$. Given the assumptions that $\dot{d}_{z_n}(\phi_\ub(t))<0$ for all $t\in[\tau_n, \tau_{n+1})$, it is clear that $|\xk_{n+1}-z_n|<r$. To determine $z_{n+1}$, one needs to solve the equation $|z_{n+1}-\xk_{n+1}|^2 = r^2$ with $z_{n+1} = \theta_{n+1}y$. Since $|\xk_{n+1}-z_n|<r$, it can be guaranteed that there exists at least one solution such that  $\theta_{n+1}>\theta_n$.  Thus, the first part of the statement  holds. It is also clear that 
    $r-|\xk_{n+1}-z_n|\leq |z_{n+1}-z_n|\leq r+|\xk_{n+1}-z_n|<2r,$ 
    which completes the proof. \pfbox

\noindent\textbf{Proof of Lemma \ref{lem: r_initial}:}
    Recall the notation $\tau:=(m+1)\dt$. For $\dt>0$ arbitrarily small, we have
    \begin{small}
            \begin{equation}\label{E: tau_soln1}
\begin{split}
        \phi_{\ub}(\tau, x_0) & = \sum_{j\in\I}\int_{j\dt}^{(j+1)\dt}(f(x_0) + G(x_0)u_{0,j})ds+ \mathcal{O}(\delta t^2)\\
        & =\dt\cdot (\sum_{j\in\I}G(x_0)u_{0,j}) + \mathcal{O}(\delta t^2)\\
        & =\tau\cdot(G(x_0)u_{0,0}) + \sum_{j\in\I_0}\dt\cdot\epsilon G(x_0)\eb_j+ \mathcal{O}(\delta t^2), 
\end{split}
\end{equation}
    \end{small}
    
\noindent where $|\mathcal{O}(\delta t^2)|\leq \int_0^\tau |f(\phi_{\ub}(t))-f(x_0)|+\|G(\phi_{\ub}(t))-G(x_0)\|dt\leq \int_0^\tau 2L_0C_0tdt=L_0M_0(m+1)^3\delta t^2=C_3\cdot\dt^2$ \cite[Lemma 4]{ornik2019control}.  In addition, by the construction of $u_{0, 0}$, we have $\tau\cdot(G(x_0)u_{0,0}) 
    =\tau(1-\epsilon)G(x_0)G(x_0)^\dagger\frac{by}{|y|}$.  Define $\underline{r}:=k\tau(b-c\rho)$.  
We then have  

\begin{small}
    \begin{equation}\label{E: initial}
    \begin{split}
       |\phi_{\ub}(\tau, x_0)|  \leq &   |\tau(1-\epsilon)\frac{Cby}{|y|}|+ |\sum_{j\in\I_0}\dt\cdot\epsilon  \ymmark{G(x_0) \eb}_j|+ |\mathcal{O}(\delta t^2) |\\
        \leq &  {C}\tau(1-\epsilon)\ymmark{b}+ \ymmark{m\cdot\dt\cdot\epsilon \cdot  \|G(x_0)\|}+ \ymmark{|\mathcal{O}(\delta t^2)|}\\
        = & \ymmark{{C}\tau(1-\epsilon)b+ \ymmark{{C}\tau\cdot\epsilon b - {C}\epsilon\delta t+ |\mathcal{O}(\delta t^2)|}}\\
        \leq & \ymmark{\underline{r}  + ({C}b\tau-\underline{r})  - C\epsilon\delta t + C_3\delta t^2}\\
       < & \ymmark{r + ({C}b\tau-\underline{r})}, 
    \end{split}
\end{equation}
\end{small}

\noindent \ymmark{where,  in the last inequality, $r\geq|\hphi_{\hub}(k\tau,x_0)|=\int_0^{k\tau}(b-c|\hphi_{\hub}(t,x_0)|)dt>\underline{r}$ holds under all $\hub=\hat{u}\in\partial\huu$,  and $-C\epsilon\delta t + C_3\delta t^2$ is canceled due to the additional assumption on its sign given in the statement. Note that for $k$ such that $k>Cb/(b-c\rho)$, it follows from \eqref{E: initial} that   $|\phi_\ub(\tau, x_0)|< r$. Additionally,   for sufficiently small $\delta t$ and $T>\frac{2kb\tau}{b-c\rho}$,  we have $2r<2\int_0^{k\tau}(b-0)dt<(b-c\rho)T<|y|$.} By solving $(z_{1}-\phi_\ub(\tau))^2 = r^2$ with $z_{1} = \theta_{1}y$ and $\theta_1\in(0, 1)$, the claim in the statement follows immediately as $|\phi_\ub(\tau)-z_0|<r$ for  $z_0=0$ and $|z_1-z_0|<2r<|y|$. \pfbox




\noindent\textbf{Proof of Proposition \ref{prop: sign_tau_n}:}

   Note that, for each $n \geq 1$, \ymmark{$\hphi_{\hub}(t, \xk_n)$ lies on the line extending from $\xk_n$ through $z_n$ and moves toward $z_n$}, under the control $\hub(t) =  \frac{z_n - \xk_n}{\lvert z_n - \xk_n \rvert}$ \ymmark{before reaching $z_n$. We also have the instantaneous Lie derivative along the path of $\hphi_{\hub}(t, \xk_n)$ at $t=0$ as} $\dot{d}_{z_n}\ymmark{(\hphi_{\hub}(0, \xk_n))} = \ymmark{\langle\nabla d_{z_n}(\hphi_{\hub}(0, \xk_n)),\; (b-c|\hphi_{\hub}(0, \xk_n)|)\hub(0)\rangle}
                 $ $=  \left\langle 2(\xk_n-z_n),\;(b-c|\xk_n|)\hub(0)\right\rangle
                 = -2r(b-c|\xk_n|)$

  Recall $\phi_\ub(\tau_n,x_0) = \xk_{n}$. By Theorem~\ref{thm: track}, there exists a $\tilde{u}$ such that $f(\xk_n) + G(\xk_n)\tilde{u} = (b - c|\xk_n|)\hub(0)$. 
  Then, $\langle\nabla d_{z_n}(\xk_n),\; f(\xk_n)+G(\xk_n)\tilde{u}\rangle 
            =  -2r(b-c|\xk_n|)$. 
Therefore, $\min_{u\in\uu} \langle\nabla d_{z_n}(\xk_n),\; f(\xk_n)+G(\xk_n)u\rangle\leq -2r(b-c|\xk_n|)$.  The optimal input   $\operatorname{argmin}_{u\in\uu}\langle\nabla d_{z_n}(\xk_n),\; f(\xk_n)+G(\xk_n)u\rangle$ satisfies the property in the statement. \pfbox

\noindent\textbf{Proof of Theorem \ref{thm: main_direction}:}
Let $n$ be fixed. 
For simplicity, we use the shorthand notation $\xb(t):=\phi_{\ub}(t, x_0)$ for any $t\in[\tau_n, \tau_{n+1})$. 
    Let $v(\xb(t)):=f(\xb(t))+G(\xb(t))\ub(t)$,  $v(\xk_n):=f(\xk_n)+G(\xk_n)\ub(\tau_n)$, 
   Then, $|v(\xb(t))-v(\xk_n)|\leq |f(\xb(t))-f(\xk_n)|+|G(\xb(t))\ub(\tau_n)-G(\xk_n)\ub(\tau_n)+G(\xb(t))(\ub(t)-\ub(\tau_n))|\leq 2L_0C_0(t-\tau_n)+\epsilon M_0$, and 
   \begin{small}
           \begin{equation}\label{E: ineq_thm}
    \begin{split}
                & \langle(\xb(t)-z_n),\;v(\xb(t))\rangle \\ 
                = & \langle \xk_n-z_n,\;v(\xk_n)\rangle   + \langle \xb(t)-\xk_n,\;v(\xk_n)\rangle\\
                & + \langle \xb(t)-\xk_n,\; v(\xb(t))-v(\xk_n)\rangle  + \langle \xk_n-z_n,\;v(\xb(t))-v(\xk_n)\rangle\\
                \leq & -r(b-c|\xk_n|)+ 2M_0|\xb(t)-\xk_n|\\
                & + |\xb(t)-\xk_n||v(\xb(t)-v(\xk_n)| + r|v(\xb(t)-v(\xk_n)|\\
                \leq & -r(b-c|\xk_n|) +\ymmarkr{2M_0C_0(t-\tau_n)}\\
                & +\ymmarkr{(C_0(t-\tau_n)+r)\cdot(2L_0C_0(t-\tau_n)+\epsilon M_0) }  \\
                < & -r(b-c|\xk_n|) + \ymmarkr{2M_0C_1\dt+(C_1\dt+r)\cdot(C_2\dt+\epsilon M_0)}\\
                = & -r(b-c|\xk_n|) + \ymmarkr{r\cdot \ee_r(\dt, \epsilon)+ \ee_n(\dt, \epsilon)}
    \end{split}
    \end{equation}
   \end{small}

\noindent 
The conclusion follows by the assumption. \pfbox

\noindent\textbf{Proof of Proposition \ref{prop: accuracy}:}
\ymmark{Recall \eqref{E: cond} and let
$\Gamma = \ee_n(\dt, \epsilon) + \ee_\mu(\dt, \epsilon) - r(b - c|\xk_n| - |a| - \ee_r(\dt, \epsilon) + \epsilon M_0)$.
Consider $\Delta_1 := (2M_0C_0 + 2rL_0C_0 + \epsilon M_0)$ and $\Delta_2(t) := 2rL_0C_0$. Then, for each $n$ and for $t \in [\tau_n, \tau_{n+1})$, by a similar argument in Theorem~\ref{thm: main_direction}, we have
$\frac{1}{2}\dot{d}{z_n}(\phi_\ub(t)) \leq \Gamma - \Delta_1 \cdot (\tau_{n+1} - t) - \Delta_2 \cdot (\tau^2 - (t - \tau_n)^2) < \Gamma \leq 0$.
Therefore,
$d_{z_n}(\phi_\ub(\tau_{n+1})) - d_{z_n}(\phi_\ub(\tau_n)) \leq 2 \int_{\tau_n}^{\tau_{n+1}} [\Gamma - \Delta_1 \cdot (\tau_{n+1} - t) - \Delta_2 \cdot (\tau^2 - (t - \tau_n)^2) ] dt = 2\Gamma \tau - \Delta_1 \tau^2 - \frac{2}{3}\Delta_2 \tau^3$.
Since $\Gamma \leq 0$, we have
$d_{z_n}(\phi_\ub(\tau_{n+1})) - d_{z_n}(\phi_\ub(\tau_n)) = |\xk_{n+1} - z_n|^2 - r^2 \leq -\Delta_1 \tau^2 - \frac{4}{3} \Delta_2 \tau^3$,
implying that $r - |\xk_{n+1} - z_n|$ has a uniform lower bound. }
    By the construction of $\set{z_n}$ in Lemma \ref{lem: z_n_geq_1}, particularly the guarantee of incremental and uniformly minimal distance for $|z_{n+1}-z_n|$ for each $n$, there exists a finite $N$ such that 
    $N=\inf\set{n\in\N: z_n\in \oball^d(y;r)}.$

\ymmarkr{By the property of the controller, specifically the attractivity to each $z_n$ under the modified condition in \eqref{E: cond_2} and \eqref{E: cond}, 
there exists a finite time $\tilde{T} \leq \tau_{N+1}$} such that $|\phi_\ub(\tilde{T}, x_0) - z_N| < r$. The conclusion follows by the triangle inequality. \pfbox

\section{Reachable Set of the Proxy System}\label{sec: proxy_extra}

We demonstrate that when $a \neq 0$, the property stated in Proposition \ref{prop: reachable_path_0} may not hold.

\begin{prop}\label{prop: reachable_path}
 Given 
 a $T$ such that $\hrr^{\leq T}( x_0)\subseteq\inte(\mathbb{B})$. Suppose $\operatorname{Im}(G(x_0))=\R^d$. For $a\neq 0$, $\partial\hrr_{\text{pse}}^{T}(x_0)\neq \partial\hrr^{T}(x_0)$.  
\end{prop}
\begin{proof}
 Due to the convexity of $\hrr^T(x_0)$, the supporting hyperplane theorem \cite{tyrrell1970convex} implies that for any non-zero $\nu\in\R^d$, there exists $x^*\in\partial\hrr^T(x_0)$  such that $\langle\nu, x^*\rangle\geq\langle \nu, x\rangle $ for all $x\in\hrr^T(x_0)$. Let $\hub^*$ be the control signal such that $\hphi_{\hub^*}(T)=x^*$, then, $\hub^*$ can be obtained by finding the maximizer of $\langle \nu, \hphi_{\hub}(T) \rangle$ subject to   \eqref{E: proxy}. Define the Hamiltonian $H(\hat{x}, \hat{u}, p)=\langle p, a + (b-c|\hat{x}|)\hat{u}\rangle$, where $p\in\R^d$ is the adjoint (co-state) vector. According to Pontryagin’s maximum principle (PMP), the optimal control $\hat{u}^*$ minimizes the Hamiltonian with respect to $\hat{u}$, leading to $|\hat{u}^*|=1$ if there exists $\hat{u}\in\hat{\uu}$ such that  $\langle p, \hat{u}\rangle\neq 0$; otherwise,   $\hat{u}^*$ is set to be anything in $\hat{\uu}$. 

Particularly, if $\operatorname{Im}G(x(0))=\R^d$, $\hat{u}^* = \frac{p}{|p|}$ if $p\neq 0$, and $\hat{u}^*$ can be arbitrary otherwise.  The optimal control law is $\hub^*(t)=\frac{p(t)}{|p(t)|}$, where the  adjoint equation is $\dot{p}(t)= \partial_{\hat{x}}H(\hphi_{\hub^*}(t), \hub^*(t), p(t)) = c\frac{|p(t)|}{|\hphi_{\hub^*}(t)|}\hphi_{\hub^*}(t)$ with $p(T) = \nu$, provided  $\hphi_{\hub^*}(t)\neq 0$). However, under the optimal control law, the set of $t$ for which $\hphi_{\hub^*}(t)\neq 0$ has measure $0$, and $p(t)$ can be determined by continuity despite the singularity.  Now, suppose $\hub^*\equiv \mathbf{e}$ is in a constant direction for some unit vector $\mathbf{e}$, then $p(t)$ is in a constant direction, which implies that    $\hphi_{\hub^*}(t)$ is in the same direction. However, this further implies that $\hphi_{\hub^*}(t) = \rho(s)w$ for a scalar function $\rho$ and a fixed vector $w\in\R^d$, which in turn requires that $a+(b-|\hphi_{\hub^*}(t)|)\mathbf{e}$  must always be consistent with $w$. This cannot hold for all $\mathbf{e}$, particularly when $a$ is not parallel to $\mathbf{e}$, as it would additionally require $|\hphi_{\hub^*}(t)|$ to remain constant (an impossible condition under $\mathbf{e}$). Then, the above argument concludes that there exist constant control inputs $\hub$ with $|\hub|\equiv 1$ that cannot steer the trajectory to the boundary. The statement is proved. 
\end{proof}

    The above statement suggests that the $\partial\hrr_{\text{pse}}^{T}(x_0)$ may lie strictly inside the true boundary. The condition of $\partial\hrr_{\text{pse}}^{T}(x_0)\neq \partial\hrr^{T}(x_0)$ for the case when $\operatorname{Im}(G(x_0))\neq\R^d$   can also be shown in a similar manner to the proof of Proposition \ref{prop: reachable_path}. However, we omit it here as it is less relevant to our current problem.  

    The following statement shows how spatial scaling of $\huu$ is related to the effect of temporal scaling.
\begin{cor}\label{cor: scaling}
    Suppose that $T$ and $k>0$ are such that $\hrr^{\leq kT}(x_0)\subseteq\inte(\B)$ and $\hrr^{\leq T}(x_0)\subseteq\inte(\B)$. Let $z$ be such that $z\in\partial\hrr^{kT}(x_0)$ 
    for some $k>0$. 
    Then, there exists a $\tilde\ub = k\hub$ such that $z \in \partial\tilde{\rr}^T(x_0)$, where $\tilde{\rr}^T(x_0)$ is the reachable set at $T$ of system \eqref{E: proxy} under $k\huu:=\set{\tilde{u}: \tilde{u}=k\hat{u}, \hat{u}\in\huu}$. \Qed
\end{cor}

\noindent\textbf{Proof of Corollary \ref{cor: scaling}}: 
By Proposition \ref{prop: reachable_path}, the controller $\hub$ is a constant vector. Note that, for each $t$, $\hphi_{\tilde{\ub}}(t)$ solves
$d\hphi_{\tilde{\ub}}(t)/dt =   (b-c|\hphi_{\tilde{\ub}}(t)|)k\hub,$
whereas $\hphi_{\hub}(kt)$ solves
$d\hphi_{\hub}(kt)/d(kt) =   (b-c|\hphi_{\hub}(kt, x)|)\hub$
in the slow time scale of $kt$. 
    Then,  $d\hphi_{\hub}(kt)/dt =   (b-c|\hphi_{\hub}(kt)|)k\hub$, which implies that $\hphi_{\tilde{\ub}}(t)=\hphi_{\hub}(kt)$ for all $t\geq 0$ and $\hub$. Letting $\hub$ be the control signal such that $\hphi_{\hub}(kt) = z\in \partial\hrr^T(x_0)$, the statement follows immediately.   \pfbox

\begin{rem}
    The uniform scaling from any ball-shaped set to a unit ball does not affect the analysis technique used to obtain the proxy system, whose reachable set is guaranteed to be a subset of the true reachable set. In view of Corollary \ref{cor: scaling}, when working with the proxy system, a uniform scaling of the input set only affects the guaranteed reachable time. \Qed
\end{rem}
    
\section{Reachability Analysis}\label{sec: reach_anal}

For $a$ that cannot satisfy \eqref{E: cond}, differences arise  in the control design at $\tau_n$. In this case, we create a reference path for $\phi_{\ub}(t) - a t$ by using controller $\frac{y-aT}{|y-aT|}$ in the proxy system to reach $y\in\hrr^T_{\text{pse}}(x_0)$. The reference path then becomes the line segment from $x_0$ to $y - aT$. We also use a piecewise constant control signal $\ub$ that ensures $\dot{d}_{z_n}(\phi_{\ub}(t, \xk_n) - a t) < 0$ for all $t \in [\tau_n, \tau_{n+1})$ and for each $n$.
In addition, we have the following inductive result as a slight modification of Lemma \ref{lem: z_n_geq_1}.

\begin{lem}\label{lem: modification}
    For each $n\geq 1$,  let  $\xk_n:=x_{n,0}-a\tau_n$ be given and consider $z_n=\theta_n(y-aT)$ with $|z_n-\xk_n|=r$ for some $\theta_n\geq 0$. Given that $\dot{d}_{z_n}(\phi_\ub(t))<a$ for all $t\in[\tau_n, \tau_{n+1})$. Then, there exists a  $z_{n+1}=\theta_{n+1}(y-aT)$ such that $\theta_{n+1}>\theta_n$ and $|z_{n+1}-\xk_{n+1}|=r$. \Qed 
\end{lem}

We summarize the algorithm in Algorithm \ref{alg: one_time2}.

\setcounter{algorithm}{1}
\begin{algorithm}[H]\label{alg: alg2}
	\caption{Control Synthesis}\label{alg: one_time2} 
	\begin{algorithmic}[2]
      \Require $d$, $m$,  $x_{0, 0}:=x_0$, $T$, $y\in\partial \hrr^T_{\text{pse}} (x_0)$, and $u_{0, 0}= (1-\epsilon) \frac{G(x_0)^\dagger (y-aT-x_0) }{\|G(x_0)^\dagger\||y-aT-x_0|}$. 
		\Require  $\dt$, $\epsilon$, $k$  based on the conditions $\epsilon > (C_3 / C)\delta t, \quad \frac{2kb\tau}{b - c\rho} < T, \quad k > \frac{Cb}{b - c\rho}$, and $\ee_n(\dt, \epsilon)+\ee_\mu(\dt, \epsilon)\leq r(b-c|\xk_n|-\ee_r(\dt, \epsilon)+\epsilon M_0)$, where $r = \sup_{\{\hub\equiv \hat{u}\in\partial\huu\}}|\phi_\ub(k(m+1)\dt, x_0)-ak(m+1)\dt-x_0|$.

  \State $n=0$.
  \Repeat 
  \State $\tau_n = n(m+1)\dt$.
  \For{$j$ \textbf{from} 0 \textbf{to} $m$}
 \State  $\ub([\tau_n + j\dt, \tau_n+(j+1)\dt)\equiv u_{n,j}$ (see Eq.~\eqref{E: u_nj};
  \State $x_{n,j+1}=\phi_\ub(\dt, x_{n,j})$.
  \EndFor
 \State  $\xk_{n+1}=x_{n,m+1}-a(\tau_{n}+(m+1)\dt)$.
 \State Determine $z_{n+1}$ based on Lemma \ref{lem: modification}. 
 \State Let 
      $u_{n+1, 0}=(1-\epsilon)\operatorname{argmin}_{u\in\uu_\lambda}\langle 2(\xk_{n+1}-z_{n+1}), \sum_{j\in\I}\lambda_j(x_{n,j+1}-x_{n,j}) \rangle,$ 
 where $\sum_{j\in\I}\lambda_j = 1$ and $\uu_\lambda$ are defined in Definition \ref{def: lambda}. 
 \State $n:=n+1$.
 \Until $n(m+1)\dt\geq T$. 
	\end{algorithmic}
\end{algorithm}

\begin{prop}
    Following Algorithm \ref{alg: one_time2}, $\phi_\ub$ will  reach $\oball^d(y; \gamma(k, \dt,\epsilon))$, where $\gamma(k, \dt,\epsilon):=N\nu(\dt)+ (rcN\nu(\dt) 
    +rC_r(\dt, \epsilon) + C_n(\dt, \epsilon) +\ee_\mu(\dt,\epsilon))\tau$, $N$ is the number of interation where the algorithm stops, $\nu(\dt)=C_1\dt-\rbar$, and $\rbar$ is the minimum distance for \eqref{E: proxy} to travel under a constant signal $\hub$ with $|\hub|\equiv 1$ within the time horizon $\tau$. \Qed
\end{prop}

\begin{proof}
For each $n$, let $\ub(t)=u_{n,j}$ for $t\in[\tau_n+j\dt, \tau_n+(j+1)\dt]$ and for $j\in\I$. 
    One can show in the same way  as in Corollary \ref{cor: decision} that $\dot{d}_{z_n}(\phi_\ub(t))<a$  for all $t\in[\tau_n, \tau_{n+1})$ and for each $n$. 

    Now we introduce $\set{\tilde{z}_n}$ with $\tilde{z}_n = \theta_n (y-aT)$ for $\theta_n\geq 0$ for the flow $\hphi_{\hub}(t)-at$ as in Lemma \ref{lem: modification}, where $\hub\equiv\frac{y-aT}{|y-aT|}$. Let $\hat{r}_n: = |\hphi_{\hub}(\tau_n)-\hphi_{\hub}(\tau_{n-1} )|$ for each $n$. We compare   the distances related to $\hphi_{\hub}(t)-at$ and $\phi_\ub(t)-at$. 
    
    Note that $|\hphi_{\hub}(t)-\phi_\ub(t)|=\mathcal{O}(\dt)$ for $t\in[0, \tau]$ and $|\tilde{z}_1-z_1|=\mathcal{O}(\dt)$. We can show inductively that, for $t\in[\tau_n, \tau_{n+1}]$, $|\hphi_{\hub}(t)-\phi_\ub(t)|\leq  n(C_1\dt-\hat{r}_n)+ n\mathcal{O}(\dt)=:n\nu(\dt)$, which implies $|\tilde{z}_n-z_n|\leq n\nu(\dt)$. 

By a direct comparison with \eqref{E: ineq_thm}, we have
\begin{equation}
\begin{split}
    &|\dot{d}_{\tilde{z}_n}(\hphi_{\hub}(t)-at)-\dot{d}_{z_n}(\phi_{\ub}(t)-at)|\\
    \leq & rc|\hphi_{\hub}(t)-\phi_\ub(t)|
    +r\cdot C_r(\dt, \epsilon)+ C_n(\dt, \epsilon) + \ee_\mu(\dt, \epsilon).
\end{split}
\end{equation}

Let $\rho(k,\dt,\epsilon,n):=rcn\nu(\dt)\tau+
    r \tau C_r(\dt, \epsilon)+ \tau C_n(\dt, \epsilon)+\tau\ee_\mu(\dt,\epsilon)$.
It follows that 
\begin{equation}
\begin{split}
    &||\hphi_{\hub}(t)-at-\tilde{z}_n|-|\phi_\ub(t)-at-z_n||\\
    \leq &\int_0^\tau|\dot{d}_{\tilde{z}_n}(\hphi_{\hub}(t))-\dot{d}_{z_n}(\phi_{\ub}(t))|dt 
    \leq   \rho(k,\dt,\epsilon,n).
\end{split}
\end{equation}

Let $N$ be such that $\tau_N\leq T\leq \tau_{N+1}$. Then, $\phi_\ub(T)=\phi_\ub(T)-z_N + z_N-z_{N-1}+z_{N-1}-z_{N-1} + \cdots   $. Due to the monotone direction of $\set{z_n}$, we have that $\phi_\ub(T)  
        =\phi_\ub(T)-z_N+\sum_{n=1}^{N}|z_{n}-z_{n-1}|\hub 
        =\phi_\ub(T)-z_N+\sum_{n=1}^{N}|\tilde{z}_{n}-\tilde{z}_{n-1}|\hub  + N(\nu(\dt))\hub 
        =  \phi_\ub(T)-z_N +\tilde{z}_N+ N\nu(\dt)\hub 
        =  \phi_\ub(T)-z_N -(\hphi_{\hub}(T) - \tilde{z}_N)+ y+  N\nu(\dt)\hub$. 
Therefore, 
\begin{equation}
    \begin{split}
       & |\phi_\ub(T)-y|\\
        \leq &||\hphi_{\hub}(t)-at-\tilde{z}_n|-|\phi_\ub(t)-at-z_n|| + N\nu(\dt)\\
        \leq & \gamma(k,\dt,\epsilon)
    \end{split}
\end{equation}
which completes the proof.
\end{proof}

\begin{rem}
The term $\gamma$ is dominated by $N\nu(\dt)$, where $N$  is proportional to $T$. The accuracy of reachability control maintains its precision only within a sufficiently small time horizon.  This limitation arises not from any deficiency in the algorithm itself, but because our aim is to leverage the knowledge of the GRS based on \eqref{E: proxy}. Unfortunately, the precision suffers due to the inherent inaccuracies in GRS estimation for large values of $a$. \Qed
\end{rem}

\end{document}